\documentclass[a4paper,10pt]{wujart}
\usepackage{amssymb}
\usepackage{a4}
\usepackage{QED}

\def\G{\Gamma}

\def\l{\lambda}

\def\m{\mu}

\def\f{\rightarrow}

\def\v{\vdash}

\def\<{\langle}
\def\>{\rangle}

\def\F{\displaystyle\frac}

\newtheorem{theorem}{Theorem}[section]
\newtheorem{lemma}{Lemma}[section]

\newtheorem{definition}{Definition}[section]
\newtheorem{notation}{Notation}[section]

\begin{document}

\begin{center}
{\LARGE\bf A short proof of the strong normalization of
the simply typed $\l\m$-calculus}\\[1cm]
\end{center}

\begin{center}
\bf
Ren\'e DAVID and Karim NOUR\\
\rm
LAMA - Equipe de Logique\\
 Universit\'e de Chamb\'ery\\
73376 Le Bourget du Lac\\
e-mail : \{david,nour\}@univ-savoie.fr\\[1cm]
\end{center}

\abstract{We give  an elementary and purely arithmetical proof of
the strong normalization of Parigot's simply typed
$\l\m$-calculus.}

\section{Introduction}

This paper gives an elementary and purely arithmetical  proof of
the strong normalization of the cut-elimination procedure for the
implicative propositional classical logic, i.e. the propositional
calculus with the connectives $\f$ and $\perp$.  As usual, $\perp$
codes the absurdity and the negation is defined by $\neg A = A \f
\perp$.

This proof is based on a proof  of the strong normalization of the
simply typed $\lambda$-calculus due to the first author (see
\cite{dav1}) which, itself, is a simplification of the one given
by R. Matthes in \cite{LoMa}. After this paper had been written we
were told by P.L. Curien and some others that this kind of
technique was already present in van Daalen (see \cite{vda}) and
J.J. Levy (see \cite{jjl}).

Since the proofs in the implicative propositional classical logic
can be coded by Parigot's  $\l\m$-terms and the cut elimination
corresponds to the  $\l\m$-reduction, the result can be seen as a
proof of the strong normalization of the simply typed
$\l\m$-calculus. The first proof of the the strong normalization
of the $\lambda\mu$-calculus for the types of Girard's system $F$
was done by Parigot in \cite{Par2} in two different ways : by
using reducibility candidates and by a CPS transformation to the
$\lambda$-calculus.

The technique we present here can also be used to prove the strong
normalization of the cut elimination procedure for the classical
natural deduction (i.e. where all the connectives, in particular
$\vee$, are present and permutative conversions are considered)
but more elaborate ideas are necessary. This result was proved
(see \cite{deG2}) by using a CPS transformation. We will give a
direct proof in a forthcoming paper.

\section{The typed system}

The $\l\m$-terms, which extend the $\l$-terms, are given by the
following grammar (where $x,y,...$ are variables):
$$
{\cal T} ::= x \mid \l x {\cal T} \mid ({\cal T} \; {\cal T})  \mid
\mu x {\cal T}
$$

\noindent The new constructor $\m$ corresponds to the classical
rule $\bot_c$ given below.

\medskip
\begin{center}

$\F{}{\G , x : A \v x : A} \, ax$ \hspace{1 cm} $\F{\G, x : A \v M
: B} {\G \v \l x M
: A \f B} \, \f_i$\\[0.5cm]

\medskip

$\F{\G_1 \v M : A \f B \quad \G_2 \v N : A} {\G_1,\G_2 \v (M \; N)
: B }\, \f_e$ \hspace{1 cm} $\F{\G , x : \neg A  \v M : \bot} {\G
\v \mu x M : A }  \, \bot_c$
 \end{center}

\medskip

\noindent The cut-elimination procedure corresponds to the
reduction rules given below.

\medskip
{\em A logical cut} appears when the introduction of the
connective $\f$ is immediately followed by its elimination. The
reduction rule is the usual $\beta$ reduction of the
$\lambda$-calculus:

$$ (\l x M \; N) \f M[x:=N] $$

{\em A classical cut } appears when the classical rule is
immediately followed by the elimination rule of $\f$.  The
reduction rule is :

$$ (\mu x M \; N) \f \mu y M[x:= \l z (y \; (z \; N))] $$

It corresponds to the following transformation on the proofs
(written in the natural deduction style):

\medskip
\begin{center}
$\F{\F{\begin{matrix}{[\neg(A \f B)] \cr {\cal D}_1 \cr \perp
\cr}\end{matrix}}{A \f B} \;\;\;
    \begin{matrix}{{\cal D}_2 \cr A \cr}\end{matrix}}{B}$
    \hspace{.5 cm} $\rightsquigarrow$ \hspace{.5 cm}
    $\F{\F{\F{\begin{matrix}{\; \cr [\neg B] \cr}\end{matrix}  \;\;\;  \F{\begin{matrix}{ \; \cr [A \f
        B] \cr}\end{matrix}  \;\;\; \begin{matrix}{{\cal D}_2 \cr A \cr}\end{matrix}}{B}}{\perp}}
{\begin{matrix}{\neg (A \f B) \cr {\cal D}_1 \cr \perp
\cr}\end{matrix}}}{B}$
\end{center}

\noindent This coding, though slightly different from the one in
\cite{Par2},  is essentially the same and the two systems are
obviously equivalent.

- Parigot uses two sets of variables: the $\l$-variables (for the
intuitionistic assumptions) and the $\mu$-variables (for classical
assumptions, i.e. the ones that are discharged by the absurdity
rule $\bot_c$). Moreover his typing judgements have several
conclusions.

- We use only one set of variables and sequents with only one
conclusion. Thus, we do not need the new constructor $[\alpha]$
and the corresponding notion of substitution. The drawback is that
the reduction introduces some ``administrative'' redexes. These
notations and reductions rules are the ones used in the
$\lambda_{\Delta}$ of Rehof and Sorensen (see \cite{ReSo}).

\section{Strong normalization}

 We first need some notations and lemmas.

\subsection{Lemmas for the un-typed calculus}

\begin{notation} Let $M$ be a $\l\m$-term.
\begin{enumerate}
\item  $M \f M'$ (resp. $M \f^* M'$)  means that $M$ reduces to $M'$
 by using one step (resp. some steps) of the reduction rules given above.
\item  $cxty(M)$ is the number of  symbols occurring in $M$.
\item $M$ is strongly normalizable (this is denoted by $M \in SN$) if there
is no infinite sequence of $\f$ reductions.
 If $M \in SN$,
 $\eta(M)$ is the length of the longest reduction of $M$.
 \item  $\overrightarrow{N}$ (resp. $\overrightarrow{\l\m}$) represents a sequence of
$\l\m$-terms (resp. of $\l$ or $\m$ abstractions). If
$\overrightarrow{N}$ is the sequence $ N_1...N_n$, $(M \;
\overrightarrow{N})$ denotes the $\l\m$-term $(M \; N_1...N_n)$.
\item In a proof by induction, $IH$ will denote the induction
hypothesis.
\end{enumerate}
 \end{notation}

\begin{lemma}\label{form}
Every $\l\m$-term $M$ can be written as $\overrightarrow{\l
\mu}(R\; \overrightarrow{O})$ where  $R$ is either a redex (called
the head-redex of $M$) or a variable (in this case, $M$ is in head
normal form).
\end{lemma}

\begin{proof}  By induction on $cxty(M)$.
\end{proof}

\begin{definition} Let $M$ be a $\l\m$-term.
\begin{enumerate}
\item $hred(M)$ is the term obtained from $M$ by reducing
  its head-redex, if any.
\item $arg(M)$ is the set of terms  defined by:
\begin{itemize}
\item $arg(\overrightarrow{\l \mu}(x\; O_1 ... O_n))
  = \{O_1,...,O_n\}$.
\item $arg(\overrightarrow{\l \mu}(\l x P \; Q\; O_1 ... O_n))=arg(\overrightarrow{\l \mu}
(\mu x P \; Q\; O_1 ... O_n))
  = \{P,Q,O_1,...,O_n\}$.
\end{itemize}
\end{enumerate}
\end{definition}

\begin{lemma}\label{arg}
Let $M,N$ be $\l\m$-terms. Then, $arg(M[x:=N])\subset arg(N) \cup
\{N\} \cup \{Q[x:=N]\; / \; Q \in arg(M)\}$.
\end{lemma}

\begin{proof} Immediate.
\end{proof}

\begin{lemma}\label{arg+head}
Let $M$ be a $\l\m$-term. Then, $M \in SN$ iff $arg(M) \subset SN$
and $hred(M)\in SN$.
\end{lemma}

\begin{proof}  $\Rightarrow$ is immediate. $\Leftarrow$ : If $M = \overrightarrow{\l \mu}(x\;
  \overrightarrow{O})$ the result is  trivial.

  - If $M =  \overrightarrow{\l \mu}(R\;
  \overrightarrow{O})$ where $R=\lambda xP\;Q$: since $arg(M) \subset SN$, an infinite
reduction of $M$ must look like:
 $M\to^*  \overrightarrow{\l \mu}(\lambda xP_{1}\;Q_{1}$\
$\overrightarrow{O_{1}}\;)\rightarrow  \overrightarrow{\l
\mu}(P_{1}[x:=Q_{1}]\; \overrightarrow{O_{1}})\to^*$ ... . The
result immediately follows from the fact that
 $(P[x:=Q]\;\overrightarrow{O} )\to^*
(P_{1}[x:=Q_{1}]\;\overrightarrow{O_{1}})$.

- If $M =  \overrightarrow{\l \mu}(R\;
  \overrightarrow{O})$ where $R=\mu xP\;Q$: the proof is similar.\qed
\end{proof}

\begin{lemma}\label{mu}
Let $M \in SN$ be $\l\m$-term. Then  $(M \, y)\in SN$.
\end{lemma}

\begin{proof} We prove by induction on $(\eta(M),cxty(M))$ that, if $M \in SN$, then
$(M[\sigma] \, y)\in SN$ where $\sigma$ is a substitution of the
form : $[x_1:=\l u
  (x_1 \; (u \; y)),...,x_n:=\l u (x_n \; (u \; y))]$. It follows immediately from the $IH$ that,
   if $N$ is a strict sub-term of $M$, then $N[\sigma] \in SN$ and thus $arg((M[\sigma] \, y)) \subset SN$.
   By lemma \ref{arg+head}, it is thus enough to prove that $N=hred((M[\sigma]
\, y))\in SN$. In each case the result follows easily from the
$IH$:
\begin{itemize}
\item If $M = (x \; \overrightarrow{O})$ and $\sigma(x) = x$: then
  $(M[\sigma] \, y) = (x \; \overrightarrow{O[\sigma]} \; y)$.
\item If $M = (x \; O_1\overrightarrow{O})$ and $\sigma(x) = \l u (x \; (u \;
y)$: then
   $N = (x \; (O_1[\sigma] \; y)
  \overrightarrow{O}[\sigma]\; y)$.
\item If $M = (\l x P \; Q\; \overrightarrow{O})$: then $N
  = (M'[\sigma] \, y)$ where $M'=hred(M)$ and thus $\eta(M') < \eta(M)$.
\item If $M = (\mu x P \; Q\; \overrightarrow{O})$: similar.
\item If $M = \l x P$: then   $N =P[\sigma][x:=y]$.
\item If $M = \mu x P$:  then   $N=\mu x P[\sigma']$
  where $\sigma' = \sigma \cup [x:=\l u (x \; (u \; y))]$. \qed
\end{itemize}
\end{proof}

\subsection{Proof of the strong normalization of the typed
calculus}
 The following result is straightforward.

\begin{lemma}
If $\G \v M : A$ and $M \f^* N$ then $\G \v N : A$.
\end{lemma}

\begin{lemma}\label{principal}
\label{b}Let  $M \in SN$ be a $\l\m$-term and $\sigma$ be a
 substitution. Assume that the substituted  variables all have  {\em the same type} and,
  for all $x$, $\sigma(x) \in SN$. Then $M[\sigma] \in SN$.
\end{lemma}

\begin{proof}
This is done by  induction on
$(lgt(\sigma),\eta(M),cxty(M),\eta(\sigma))$ where $lgt(\sigma)$
is the number of connectives in the type of the substituted
variables and $\eta(\sigma)$ is the sum of the $\eta(N)$ for the
$N$ that are actually substituted, i.e. for example if
$\sigma=[x:=N]$ and $x$ occurs $n$ times in $M$, then
$\eta(\sigma)=n.\eta(N)$.
 The cases $M=\lambda xP$,
$M=\mu xP$ and $M=(y\; \overrightarrow{P})$ for $y\neq x$ are
trivial. Otherwise, by the $IH$ and lemma \ref{arg},
$arg(M[\sigma]) \subset SN$. By lemma \ref{arg+head} it is thus
enough to show that $hred(M[\sigma]) \in SN$:

\begin{itemize}

\item If $M=(\lambda y P\;Q\;\overrightarrow{O})$ :
$hred(M[\sigma])=hred(M)[\sigma]$  and the result follows from the
$IH$ since $\eta(hred(M))<\eta(M)$.

\item $M=(\mu y P\;Q\;\overrightarrow{O})$ : similar.

\item $M=(x\;P\;\overrightarrow{O})$ : by our
definition of $\eta(\sigma)$, we may assume, without loss of
generality, that $x$ occurs only once in $M$. Let $N=\sigma(x)$.

\begin{itemize}

\item If $N$ is not in head normal form,
$hred(M[\sigma])=M[\sigma']$ where $\sigma'(y)=\sigma(y)$ for $y
\neq x$ and   $\sigma'(x)=hred(N)$. The result follows from the
$IH$ since $\eta(\sigma') < \eta(\sigma)$.

\item If $N=(y \; \overrightarrow{N_1})$, the result is trivial.

\item If $N = \l y N_{1}$ then
$hred(M[\sigma])=(N_{1}[y:=P[\sigma]]\;\overrightarrow{O[\sigma]})$.
By the $IH$, since $lgt(P[\sigma])<lgt(\sigma)$,
$N_{1}[y:=P[\sigma]] \in SN$ and thus, by the $IH$,
$hred(M[\sigma])=(z\;\overrightarrow{O[\sigma]})\;[z:=N_{1}[y:=P[\sigma]]]
\in SN$ since $lgt(N_{1})<lgt(\sigma)$.

\item If $N = \mu y N_1$ then $M[\sigma]=(\mu y N_1 \,P[\sigma] \, \overrightarrow{O[\sigma]})$.  Let $M_1 = (\mu y N_1 \, z)$ where $z$ is a fresh variable. By
lemma \ref{mu}, $M_1 \in SN$. Since $lgt(P[\sigma]) <
lgt(\sigma)$, by the $IH$, $(\mu y N_1 \, P[\sigma]) = M_1[z :=
P[\sigma]]\in SN$. But $M[\sigma]=M_2[u:= (\mu y N_1 \,
P[\sigma])]$ where $M_2=(u \,\overrightarrow{O[\sigma]})$ and $u$
is a fresh variable. Since $lgt((\mu y N_1 \,
P[\sigma]))<lgt(\sigma)$, the result follows from the the $IH$.
\qed
\end{itemize}
\end{itemize}
\end{proof}

\begin{theorem}\label{SN}
Every typed $\l\m$-term is strongly normalizable.
\end{theorem}

\begin{proof}  By induction on $cxty(M)$. The cases
$M=x$, $M=\lambda x\;N$ or $M=\mu x\;N$ are trivial. If
$M=(N_1\;N_2)$ the result follows from lemma \ref{b} and the $IH$
since $M=(x\;N_2)[x:=N_1]$ where $x$ is a fresh variable.
\end{proof}

\noindent{\bf  Remarks}
\begin{enumerate}
\item In the proof of theorem \ref{SN}, the case $M = (N_1\;N_2)$ can also be solved,
  by writing $M=(N_1 \; x)[x:=N_2]$ where $x$ is a fresh variable
 and using lemma \ref{arg+head}.
\item In the proof of lemma \ref{mu}, the case $M =
  (x\;P\;\overrightarrow{O})$ and $N = \l y N_{1}$ can be solved exactly as the case $N = \m y N_{1}$
   by using  $M_1 = (\l y N_1 \, z)$ where $z$ is a fresh
variable.
\end{enumerate}

\end{document}